\documentclass[a4paper,11pt]{article} 
\usepackage{amsthm}
\usepackage{amsthm, amsmath, amssymb, url,tikz}
\usepackage{hyperref}
\usepackage{pgfplots}
\pgfplotsset{compat=newest}
\usepackage{geometry}
\usepackage{mathrsfs}
\usepackage{caption}
\usepackage{graphicx}
\newtheorem{theorem}{Theorem}[section]
\newtheorem{lemma}[theorem]{Lemma}

\newtheorem{proposition}[theorem]{Proposition}
\newtheorem{corollary}[theorem]{Corollary}

\newtheorem{example}{Example}

\newtheorem{problem}{Problem}

\usepackage[numbers,sort&compress]{natbib}

\title{Graph Sensitivity of Cartesian Products with Matched Bridges}

\author{Zhen-Mu Hong$^{\,\rm a}$, Zi-Yi Wu$^{\,\rm a}$, Zheng-Jiang Xia$^{\,\rm b}$\thanks{Corresponding author. 
E-mail addresses: zmhong@mail.ustc.edu.cn (Z.-M. Hong), wuziyiaufe@163.com (Z.-Y. Wu), xzj@mail.ustc.edu.cn (Z.-J. Xia).}\\
{\footnotesize$^{\rm a}$School of Statistics and Applied Mathematics, Anhui University of Finance \& Economics,} \\{\footnotesize Bengbu 233030, China}\\
{\footnotesize$^{\rm b}$School of Finance, Anhui University of Finance \& Economics, Bengbu 233030, China}}

\date{}

\begin{document}

\maketitle

\parbox{14.5cm}{
\begin{center}
\textbf{Abstract}
\end{center}

For a graph $G$, let $f_t(G)$ denote the minimum of the maximum degree of an induced subgraph with $\alpha(G)+t$ vertices, where $\alpha(G)$ is the independence number, and write $f(G)=f_1(G)$. Huang’s theorem gives $f(Q_k)\ge\lceil\sqrt{k}\rceil$ for the $k$-dimensional hypercube $Q_k$. We extend this lower bound to Cartesian products of $k$ bipartite graphs with perfect matchings and prove that equality holds when the factors are connected and each has a matched bridge. In particular, we prove that $f(T_1\square\cdots\square T_k)=\lceil\sqrt{k}\rceil$ whenever each $T_i$ is a tree with a perfect matching. We determine the sensitivity of every Cartesian product of paths, settling the even-path case left open by Zeng and Hou [J. Graph Theory
107 (2024), 169--180]. For these tree products, with  $D=\lceil\sqrt{k}\rceil$, we also prove that $f_t(T_1\square\cdots\square T_k)=D$ whenever $1\le t\le 2^{D-\lceil\log_2D\rceil-1}$. When $t=2$, this equality holds for all $k\ge 2$, provided that at least one factor is not $K_2$. Matching cuts give an additional exact range for products of even-order paths. Finally, we prove that $f_2(Q_k)=\lceil\sqrt{k}\rceil$ for every $k\ge 2$ with $k\notin \{4,9\}$, whereas $f_2(Q_4)=3$ and $3\le f_2(Q_9)\le 4$.


}

\section{Introduction}

All graphs considered are finite and simple. For $U\subseteq V(G)$, let $G[U]$ denote the subgraph induced by $U$, and let $\alpha(G)$ and $\Delta(G)$ denote the independence number and maximum degree of $G$, respectively. We write $P_n$ for the path on $n$ vertices, $G_1\square G_2$ for the Cartesian product of $G_1$ and $G_2$, and $G^{\square k}$ for the Cartesian product of $k$ copies of $G$. Following a question of Huang~\cite{Huang2019}, for a graph $G$ containing at least one edge, define
\begin{equation}
 f(G)=\min\{\Delta(G[U]):U\subseteq V(G),\ |U|=\alpha(G)+1\}.
 \label{eq:def-f}
\end{equation}
The parameter measures the smallest maximum degree that can occur in a vertex set whose size exceeds the independence number by one.  For the $k$-dimensional
hypercube $Q_k=K_2^{\square k}$, the construction of Chung, F\"uredi, Graham, and
Seymour~\cite{CFGS1988}, together with Huang's theorem~\cite{Huang2019},
gives
\begin{equation}
 f(Q_k)=\lceil\sqrt{k}\rceil.
 \label{eq:cube-value}
\end{equation}
The lower bound in \eqref{eq:cube-value} is the graph-theoretic core of
Huang's proof of the Sensitivity Conjecture.  The quantity in
\eqref{eq:def-f} is also called the \textit{graph sensitivity} of $G$; see
Garc\'ia-Marco and Knauer~\cite{GaMK2022}.

The same question has subsequently been considered for Cayley graphs,
Hamming graphs, oriented cycles, and several graph products; see, for
example,~\cite{AlZh2020,HoLL2021,Tika2022,
AFGK2026,KrTo2026}.  Zeng and Hou~\cite{ZeHo2024}
studied the Cartesian powers of paths.  They proved
$f(P_3^{\square k})=2$, $f(P_{2r+1}^{\square k})=1$ for $r\ge 2$,
and obtained a spectral lower bound for $P_{2r}^{\square k}$, that is,
$$
f(P_{2r}^{\square k})\ge \left\lceil
  2\sin\!\left(\frac{\pi}{4r+2}\right)\sqrt{k}
 \right\rceil.
$$
They asked whether that lower bound is sharp when $r\ge3$.

Our first main result answers this question and determines the exact value.

\begin{theorem}\label{thm:even-paths-intro}
For all integers $r_1,\ldots,r_k\ge1$,
$$
 f(P_{2r_1}\square\cdots\square P_{2r_k})
 =\lceil\sqrt{k}\rceil.
$$
In particular,
$f(P_{2r}^{\square k})=\lceil\sqrt{k}\rceil$ 
for $r,k\ge 1$.

\end{theorem}

The proof has two complementary parts.  The lower bound follows by partitioning the product into induced copies of $Q_k$ using perfect matchings of the factors. A set containing more than half of the product vertices must contain more than half of the vertices of one copy, so Huang's theorem applies. For the upper bound, we lift a cube coloring to the whole product using a binary map in each factor, while controlling edges between different copies.

The mechanism applies to much more than paths.  We call an edge a
\emph{matched bridge} if it is a bridge that belongs to a perfect matching.

\begin{theorem}\label{thm:trees-intro}
If $T_1,\ldots,T_k$ are trees with perfect matchings, then
$$
 f(T_1\square\cdots\square T_k)=\lceil\sqrt{k}\rceil.
$$
\end{theorem}

In a tree with a perfect matching, every edge of that matching is a bridge. Thus Theorem~\ref{thm:trees-intro} follows from the more general product theorem in Section~\ref{sec:product}, which also allows additional unbalanced bipartite factors; see
Theorem~\ref{thm:general-product}.

Combining Theorem~\ref{thm:even-paths-intro} with the odd-path constructions gives a complete answer for Cartesian products of paths, even when the paths have different numbers of vertices.

\begin{theorem}\label{thm:all-path-intro}
Let $n_1,\ldots,n_k\ge 2$ be integers, let
$G=P_{n_1}\square\cdots\square P_{n_k}$,
and let $e$ be the number of even integers among $n_1,\ldots,n_k$.  Then
$$
 f(G)=
 \begin{cases}
  \lceil\sqrt e\rceil, & e\ge1,\\
  2, & e=0\text{ and }n_1=\cdots=n_k=3,\\
  1, & e=0\text{ and at least one }n_i\ge5.
 \end{cases}
$$
\end{theorem}

There is a direct connection with dissociation sets.  If
$\operatorname{diss}(G)$ denotes the largest order of an induced subgraph of
maximum degree at most one, then
$$
 f(G)=1 \text{ if and only if }
 \operatorname{diss}(G)\ge \alpha(G)+1.
$$
Moreover, $\operatorname{diss}(G)=|V(G)|-\psi_3(G)$, where $\psi_3$ is the
minimum size of a vertex subset meeting every path of order three (see~\cite{BJKST2013,JaTa2013,LiZu2018}).  Exact formulas for $\psi_3(P_m\square P_n)$ obtained by Bre\v{s}ar et al.~\cite{BJKST2013}, combined with elementary
constructions, therefore already determine
$f(P_m\square P_n)$. The two-dimensional cases are included below for completeness. The results here extend the formula to arbitrary products of paths of different orders and place it within a product theorem for bipartite factors with matched bridges.

We next consider vertex sets whose size exceeds the independence number by a prescribed positive integer $t$. For an integer $t\ge1$ for which $\alpha(G)+t\le |V(G)|$, put
\begin{equation}
 f_t(G)=\min\{\Delta(G[U]):U\subseteq V(G),\ |U|=\alpha(G)+t\}.
 \label{eq:def-ft}
\end{equation}
This fixed-excess question was recorded as Question~5.1 by Zeng and Hou~\cite{ZeHo2024},
following an oral question of Huang.  It is called \textit{$k$-sensitivity} in
Definition~2.7 of Kriloff and Tolman \cite{KrTo2026}, with their $k$ playing the role of our excess $t$. Our construction determines $f_t$ for these tree products whenever the number of factors is sufficiently large relative to $t$, with the following explicit range.

\begin{theorem}\label{thm:fixed-excess-intro}
Let $k\ge1$, let $T_1,\ldots,T_k$ be trees with perfect matchings, and put $D=\lceil\sqrt{k}\rceil$. 
\begin{enumerate}
 \item[$(i)$] If $1\le t\le 2^{D-\lceil\log_2D\rceil-1}$, then $f_t(T_1\square\cdots\square T_k)=\lceil\sqrt{k}\rceil$.
 \item[$(ii)$] If $k\ge2$ and at least one $T_i$ is not $K_2$, then $f_2(T_1\square\cdots\square T_k)=\lceil\sqrt{k}\rceil$.
\end{enumerate}
\end{theorem}

Thus, for every fixed $t$, Theorem~\ref{thm:fixed-excess-intro} determines $f_t$ for all products of sufficiently many trees with perfect matchings, uniformly over the choice of the factors.  A second construction, based on matching
cuts, gives exact values for even paths throughout the range
$$
 f_t(P_{2r_1}\square\cdots\square P_{2r_k})
 =\lceil\sqrt{k}\rceil
 \quad\text{whenever}\quad
1\le t\le\prod_{i=1}^k\left\lceil\frac{r_i}{3}\right\rceil.
$$
For hypercubes of dimension $k\ge2$, we obtain
$$
 f_2(Q_k)=\lceil\sqrt{k}\rceil \text{ for } k\ne4,9, 
 \quad f_2(Q_4)=3,
$$
while for the remaining case $k = 9$ we have
$3\le f_2(Q_9)\le4$.

The paper is organized as follows. Section~\ref{sec:cube} introduces
Boolean sensitivity and the cube coloring used in our constructions.
Section~\ref{sec:product} establishes the cube-partition lower bound and
the matched-bridge product theorem, proving Theorem~\ref{thm:trees-intro}. Section~\ref{sec:paths} determines
the sensitivity of arbitrary Cartesian products of paths, which yields Theorem~\ref{thm:all-path-intro}.
Section~\ref{sec:fixed-excess} develops the matching-cut and separated-set
constructions for fixed excess, proving Theorem~\ref{thm:fixed-excess-intro}.
Section~\ref{sec:cube-excess-two} treats the $2$-sensitivity of the hypercube.
Section~\ref{sec:discussion} discusses the scope of the methods and the
remaining questions.

\section{Boolean sensitivity and hypercube colorings}\label{sec:cube}

For $x=(x_1,\ldots,x_n)\in\{0,1\}^n$, let
$|x|=\sum_{i=1}^n x_i$ denote its Hamming weight, and let
$x^{(i)}$ be the vector obtained by flipping its $i$-th coordinate. Write $[n]=\{1,\ldots,n\}$ and $\operatorname{supp}(x)=\{i\in[n]:x_i=1\}$.
For a Boolean function $F:\{0,1\}^n\to\{0,1\}$, its sensitivity
at $x$ and its sensitivity are defined by
$$
 s(F,x)=\bigl|\{i\in[n]:F(x^{(i)})\ne F(x)\}\bigr|,
 \quad
 s(F)=\max_{x\in\{0,1\}^n}s(F,x).
$$
Thus $s(F,x)$ counts the neighbors of $x$ in $Q_n$ at which
$F$ takes a different value. We also write
$$
 s_\varepsilon(F)
 =\max_{\substack{x\in\{0,1\}^n\\F(x)=\varepsilon}}s(F,x),
 \quad \varepsilon\in\{0,1\},
$$
where the maximum over an empty set is taken to be zero. Thus $s(F)=\max\{s_0(F),s_1(F)\}$.

Define the alternating sum of $F$ by
$$
 A(F)=\sum_{x\in\{0,1\}^n}(-1)^{|x|}F(x).
$$
For Boolean functions $F$ and $H$ on disjoint sets of variables,
their conjunction is given by $(F\wedge H)(x,y)=F(x)H(y)$.
Summing separately over the two sets of variables gives
$$
 A(F\wedge H)=A(F)A(H).
$$
Moreover,
$$
 s_0(F\wedge H)\le\max\{s_0(F),s_0(H)\},
 \quad
 s_1(F\wedge H)\le s_1(F)+s_1(H).
$$
Indeed, if both factors vanish, changing one coordinate cannot change
their conjunction. If exactly one factor vanishes, only changes in
that factor can affect the output. If both factors equal one, the
sensitive coordinates come from the two factors separately.

A basic example is the Boolean OR function
$$
 \operatorname{OR}_n(x)=1-\prod_{i=1}^n(1-x_i),\quad n\ge1,
$$
which is zero precisely at the all-zero vector. It satisfies
$$
 s_0(\operatorname{OR}_n)=n,\quad
 s_1(\operatorname{OR}_n)=1,\quad A(\operatorname{OR}_n)=-1.
$$

Write $x\sim y$ when $x$ and $y$ are adjacent in $Q_k$.  If
$\chi:V(Q_k)\to\{-1,1\}$, the \emph{monochromatic degree} of $x$ is
$$
 d_\chi(x)=|\{y:y\sim x,\ \chi(y)=\chi(x)\}|.
$$
The \emph{imbalance} of a coloring $\chi$ is the absolute difference between the sizes of its two color classes, equivalently $\left|\sum_x\chi(x)\right|$.

The following observation converts Boolean functions into cube colorings.

\begin{lemma}\label{lem:parity-twist-excess}
Let $n\ge1$ and let $F:\{0,1\}^n\to\{0,1\}$. Define
$\chi_F(x)=(-1)^{|x|+F(x)}$. Then
$$
 d_{\chi_F}(x)=s(F,x)\quad\text{for every }x,
 \quad \sum_x\chi_F(x)=-2A(F).
$$
Consequently, one color class has $2^{n-1}+|A(F)|$ vertices and induces a subgraph of maximum degree at most $s(F)$. If $A(F)\ne0$, then
$f_{|A(F)|}(Q_n)\le s(F)$.
\end{lemma}

\begin{proof}
Across a cube edge the parity of $|x|$ changes. The endpoints therefore
receive the same $\chi_F$-color exactly when $F$ changes value across the
edge, proving the degree assertion. Since $(-1)^{F(x)}=1-2F(x)$ and the
parity character has sum zero,
$$
 \sum_x\chi_F(x)
 =\sum_x(-1)^{|x|}\bigl(1-2F(x)\bigr)=-2A(F).
$$
The two color classes thus differ in size by $2|A(F)|$, proving the
remaining assertions.
\end{proof}

The next lemma is the two-sided form of the classical construction
of Chung, F\"uredi, Graham, and Seymour~\cite{CFGS1988}, in which both color classes have small induced maximum degree and their sizes differ by two; see also
Asensio et al.~\cite{AFGK2026}. We include the proof for completeness.

\begin{lemma}\label{lem:cube-coloring}
For every $k\ge1$, there is a function
$\chi:\{0,1\}^k\to\{-1,1\}$ such that
$$
 \sum_{x\in\{0,1\}^k}\chi(x)=2
 \quad\text{and}\quad
 d_\chi(x)\le\lceil\sqrt{k}\rceil
 \quad\text{for every }x.
$$
\end{lemma}

\begin{proof}
Put $D=\lceil\sqrt{k}\rceil$ and
$q=\lceil k/D\rceil$.  Then $q\le D$.  Partition $[k]$ into nonempty sets
$I_1,\ldots,I_q$, each of size at most $D$, and define
$$
 F(x)=\prod_{j=1}^q
 \left(1-\prod_{i\in I_j}(1-x_i)\right),
 \quad x\in\{0,1\}^k.
$$
Thus $F(x)=1$ exactly when every block $I_j$ contains a coordinate equal to $1$.

If $F(x)=0$, flipping a coordinate can change $F(x)$ only when exactly one block is all zero; that block has at most $D$ coordinates. If $F(x)=1$, a sensitive coordinate must be the unique $1$ in its block, so there are at most $q$ such coordinates. Hence, $s(F)\le \max\{D,q\}=D$.

Since $F$ is the conjunction of the OR functions on the disjoint blocks,
$$
 A(F)=\prod_{j=1}^q A(\operatorname{OR}_{|I_j|})=(-1)^q.
$$
Set $\chi=(-1)^{q+1}\chi_F$. Lemma~\ref{lem:parity-twist-excess} gives
$$
 \sum_x\chi(x)=(-1)^{q+1}(-2A(F))=2,
 \quad d_\chi(x)=s(F,x)\le D.
$$
\end{proof}

The positive color class in Lemma~\ref{lem:cube-coloring} has
$2^{k-1}+1$ vertices, the negative class has $2^{k-1}-1$, and both induced
subgraphs have maximum degree at most $\lceil\sqrt{k}\rceil$.  This is a
convenient two-sided formulation of the construction underlying the sharp
upper bound for the cube.

We shall also use the following theorem of Huang~\cite{Huang2019}.

\begin{theorem}\label{thm:huang}
If $U\subseteq V(Q_k)$ and $|U|\ge 2^{k-1}+1$, then
$$
 \Delta(Q_k[U])\ge\lceil\sqrt{k}\rceil.
$$
\end{theorem}

\section{Cartesian products with matched bridges}\label{sec:product}

We first present the lower bound supplied by perfect matchings.

\begin{lemma}\label{lem:cube-cell-bound}
Let $k\ge1$ and $\ell\ge0$. Let $G_1,\ldots,G_k$ be bipartite graphs with perfect matchings,
and let $H_1,\ldots,H_\ell$ be bipartite graphs. Assume that every
factor has at least one vertex.
Put
$$
 G=G_1\square\cdots\square G_k
   \square H_1\square\cdots\square H_\ell,
 \quad N=|V(G)|.
$$
Then $\alpha(G)=N/2$, and
$$
 f_t(G)\ge\lceil\sqrt{k}\rceil
 \quad\text{for }1\le t\le N/2.
$$
\end{lemma}

\begin{proof}
Replicating a perfect matching of $G_1$ over all choices of the other
coordinates gives a perfect matching of $G$. An independent set contains
at most one endpoint of each matching edge, so $\alpha(G)\le N/2$.
Since $G$ is bipartite, the reverse inequality also holds.

Fix a perfect matching in each $G_i$. Products of their edges, together
with a fixed vertex in each $H_j$, partition $V(G)$ into induced copies
of $Q_k$. Every set $U$ with $|U|=N/2+t>N/2$ contains more than half the
vertices of at least one such copy. Theorem~\ref{thm:huang} therefore gives
$\Delta(G[U])\ge\lceil\sqrt{k}\rceil$.
\end{proof}

A bipartite graph $G$ with a perfect matching is balanced.  Fix a bipartition
character $\rho:V(G)\to\{-1,1\}$, so that $\rho$ changes sign across every edge. The next lemma constructs a binary map that is constant on each side of a matched bridge.

\begin{lemma}\label{lem:bridge-gate}
Let $G$ be a connected bipartite graph with a perfect matching $M$, and let
$uv\in M$ be a bridge. There are functions
$$
 \tau:V(G)\to\{0,1\},\quad
 \eta:V(G)\to\{-1,1\}
$$
with the following properties:
\begin{enumerate}
 \item[$(i)$] $\sum_{\tau(x)=0}\eta(x)=\sum_{\tau(x)=1}\eta(x)=1$;
 \item[$(ii)$] for every edge $xy\in E(G)\setminus\{uv\}$,
 $\tau(x)=\tau(y)$ and $\eta(x)=-\eta(y)$;
 \item[$(iii)$] $\eta(u)=\eta(v)$ and $\tau(u)\ne\tau(v)$.
\end{enumerate}
\end{lemma}

\begin{proof}
Let $L$ and $R$ be the components of $G-uv$ containing $u$ and $v$, respectively.  Reverse $\rho$ if necessary so that $\rho(u)=1$ and $\rho(v)=-1$.  The matching $M-uv$ pairs all vertices of $L$ except $u$ and all vertices of $R$ except $v$. Since the endpoints of each matching edge have opposite $\rho$-values,
$$
 \sum_{x\in L}\rho(x)=1,
 \quad
 \sum_{x\in R}\rho(x)=-1.
$$
Set $\tau=0$ on $L$, $\tau=1$ on $R$, and
$\eta(x)=(-1)^{\tau(x)}\rho(x)$.
Then $\sum_{x\in L}\eta(x)=\sum_{x\in R}\eta(x)=1$ and $\eta(u)=\eta(v)=1$.  Inside either component, $\tau$ is
constant and $\eta$ changes across every edge, whereas both $\rho$ and $\tau$ change across $uv$.  This proves all three assertions.
\end{proof}

The coordinate-sum lift of Asensio et al.~\cite[Lemma~2.3]{AFGK2026}
is designed for Cartesian powers. Here the matched-bridge construction
allows a separate binary map in each factor, including nonisomorphic
factors, with at most one incident edge in each factor contributing to
the monochromatic degree.

We can now state the general product theorem.  A bipartite graph will be
called \emph{unbalanced} if it has a bipartition whose two classes have
different sizes.

\begin{theorem}\label{thm:general-product}
Let $G_1,\ldots,G_k$ be connected bipartite graphs, each having a perfect matching that contains a bridge.  Let $H_1,\ldots,H_\ell$ be unbalanced
bipartite graphs, where $\ell\ge 0$, and put
$$
 G=G_1\square\cdots\square G_k
   \square H_1\square\cdots\square H_\ell.
$$
If $k\ge 1$, then
$f(G)=\lceil\sqrt{k}\rceil$.
\end{theorem}

\begin{proof}
Let $N=|V(G)|$. By Lemma~\ref{lem:cube-cell-bound},
$\alpha(G)=N/2$ and $f(G)\ge\lceil\sqrt{k}\rceil$.

For the upper bound, apply Lemma~\ref{lem:bridge-gate} to each $G_i$, and
write the resulting functions as $\tau_i$ and $\eta_i$.  For every $H_j$,
choose a bipartition character $\zeta_j$ with positive sum; this is possible
because $H_j$ is unbalanced.  Take the cube coloring $\chi$ from
Lemma~\ref{lem:cube-coloring} and for each $v=(x_1,\ldots,x_k,y_1,\ldots,y_\ell)\in V(G)$ define
$$
 \Psi(v)=\Psi(x_1,\ldots,x_k,y_1,\ldots,y_\ell)
  = \chi(\tau_1(x_1),\ldots,\tau_k(x_k))
    \cdot\prod_{i=1}^k\eta_i(x_i)
    \prod_{j=1}^\ell\zeta_j(y_j).
$$
For $a_i\in\{0,1\}$, let 
$\tau_i^{-1}(a_i)=\{x_i\in V(G_i):\tau_i(x_i)=a_i\}$ be the fiber of $\tau_i$ over $a_i$. By Lemma~\ref{lem:bridge-gate},
$\sum_{x_i\in\tau_i^{-1}(a_i)}\eta_i(x_i)=1$. 
Grouping the sum according
to $a=(\tau_1(x_1),\ldots,\tau_k(x_k))$ therefore gives
\begin{align*}
 \sum_{v\in V(G)}\Psi(v)
 &=\left(\sum_{a\in\{0,1\}^k}\chi(a)
     \prod_{i=1}^k\sum_{x_i\in\tau_i^{-1}(a_i)}\eta_i(x_i)\right)
   \prod_{j=1}^{\ell}\sum_{y\in V(H_j)}\zeta_j(y)\\
 &=2\prod_{j=1}^{\ell}\sum_{y\in V(H_j)}\zeta_j(y)\ge2.
\end{align*}
Here each $\zeta_j$ has positive integral sum, and an empty product is one.
Thus the positive color class
$$
 U_+=\{v\in V(G):\Psi(v)=1\}
$$
has at least $N/2+1=\alpha(G)+1$ vertices.

It remains to bound its induced degree.  An edge in an $H_j$-direction
changes $\zeta_j$ and hence changes the sign of $\Psi$.  An edge in a $G_i$-direction other than the selected bridge
leaves $\tau_i$ unchanged and changes $\eta_i$, so it also changes the sign
of $\Psi$.  Across the selected bridge of $G_i$, the
$\eta_i$-value stays fixed while the $i$-th input of $\chi$ changes.  Such an
edge is monochromatic precisely when the corresponding cube edge is
monochromatic under $\chi$.  At a fixed product vertex these edges map to
distinct cube directions.  Lemma~\ref{lem:cube-coloring} therefore gives
$$
 \Delta(G[U_+])\le\lceil\sqrt{k}\rceil.
$$
Delete vertices from $U_+$ if necessary until exactly $N/2+1$ remain.
Deletion does not increase maximum degree, so the lower and upper bounds
agree.
\end{proof}

\begin{proof}[Proof of Theorem~\ref{thm:trees-intro}]
Every edge of a tree is a bridge.  Choose any edge from a perfect matching
of each $T_i$ and apply Theorem~\ref{thm:general-product} with no additional unbalanced factors.
\end{proof}

\section{Cartesian products of paths}\label{sec:paths}

We first record the independence number of an arbitrary Cartesian product of paths.

\begin{lemma}\label{lem:alpha-box}
If $n_1,\ldots,n_k\ge 2$ and $N=\prod_{i=1}^k n_i$, then
$$
 \alpha(P_{n_1}\square\cdots\square P_{n_k})=\left\lceil\frac{N}{2}\right\rceil.
$$
\end{lemma}

\begin{proof}
Let $G=P_{n_1}\square\cdots\square P_{n_k}$.
Color a vertex $(x_1,\ldots,x_k)$ by the parity of $x_1+\cdots+x_k$. The larger color class gives $\alpha(G)\ge\lceil N/2\rceil$.

To prove the reverse inequality, we construct a matching that
covers all but at most one vertex. For every fixed choice of
$x_1,\ldots,x_{k-1}$, pair the vertices in the last coordinate
using the edges $(1,2),(3,4),(5,6),\ldots$.
If $n_k$ is even, this matches every vertex of $G$. If $n_k$
is odd, the unmatched vertices are precisely those with
$x_k=n_k$, and they induce a copy of
$P_{n_1}\square\cdots\square P_{n_{k-1}}$.
We then apply the same procedure to this remaining
$(k-1)$-dimensional product. Repeating recursively, we obtain a
matching $M$ that leaves at most one vertex unmatched. Thus
$|M|=\left\lfloor\frac{N}{2}\right\rfloor$.

An independent set contains at most one endpoint of each edge of
$M$, together with the possible unmatched vertex. Therefore,
$$
\alpha(G) \leq |M|+N-2|M|
=N-|M|
=\left\lceil\frac{N}{2}\right\rceil.
$$
The result follows.
\end{proof}

The case containing at least one even factor is immediate from the general
product theorem.

\begin{proposition}\label{prop:mixed-box}
Let $e\geq 1$ denote the number of even integers among
$n_1,\ldots,n_k$. Then
$$
 f(P_{n_1}\square\cdots\square P_{n_k})
 =\lceil\sqrt e\rceil.
$$
\end{proposition}

\begin{proof}
Every even path has a perfect matching, every edge of a path is a bridge, and every odd path is unbalanced.  Apply
Theorem~\ref{thm:general-product}, taking the $e$ even paths as the $G_i$ and the odd paths as the $H_j$.
\end{proof}

For completeness, we give short proofs for the all-odd cases.  The following
sign-word construction applies whenever one factor has order at least
$5$.

\begin{proposition}\label{prop:odd-large}
Suppose that every $n_i$ is odd and at least one $n_i$ is at least $5$.
Then
$$
 f(P_{n_1}\square\cdots\square P_{n_k})=1.
$$
\end{proposition}

\begin{proof}
Put $G=P_{n_1}\square\cdots\square P_{n_k}$ and $N=|V(G)|$.
Write $n_1=2t+1\ge5$ after reordering the factors. We assign signs to the vertices of the first path $P_{n_1}$ in their order along the path, namely
\begin{equation}
 ++-++(-+)^{t-2},
 \label{eq:sign-word}
\end{equation}
where $(-+)^{t-2}$ means $t-2$ repetitions of the pair $-+$.
Its sign sum is $3$, and each monochromatic component has at most two
vertices.  On every other odd path use its alternating bipartition character,
whose sign sum has absolute value $1$.  Color the product by the product of
these coordinate signs.  The total sign sum has absolute value $3$, so the
larger color class has
$$
 (N+3)/2=\alpha(G)+1
$$
vertices by Lemma~\ref{lem:alpha-box}.  Monochromatic edges can occur only in
the first coordinate, and the sign word \eqref{eq:sign-word} gives induced
maximum degree at most $1$.  Since a set of $\alpha(G)+1$ vertices cannot be
independent, equality follows.
\end{proof}

It remains to treat products of $P_3$.  Zeng and Hou already determined this case~\cite{ZeHo2024}; we include a short proof using anti-commuting signed adjacency matrices, so that the classification is self-contained.

\begin{proposition}\label{prop:p3}
For every $k\ge 1$, $f(P_3^{\square k})=2$.
\end{proposition}

\begin{proof}
Let $A$ be the adjacency matrix of $P_3$ and put
$S=\operatorname{diag}(1,-1,1)$.  Then $SA=-AS$.  Define the signed adjacency matrix of $P_3^{\square k}$ by 
$$
 M_k=\sum_{i=1}^k
 S^{\otimes(i-1)}\otimes A\otimes I_3^{\otimes(k-i)}.
 $$
The summands anticommute pairwise, so all cross terms vanish
when $M_k$ is squared. Thus, 
$$
 M_k^2=\sum_{i=1}^k
 I_3^{\otimes(i-1)}\otimes A^2\otimes I_3^{\otimes(k-i)}.
$$
The eigenvalues of $A^2$ are $0,2,2$.  It follows that $M_k$ has a
one-dimensional kernel and that every nonzero eigenvalue has absolute value
at least $\sqrt 2$.  In addition, $M_k$ anticommutes with $S^{\otimes k}$. If $\lambda$ is an eigenvalue of $M_k$ with eigenvector $\xi$, then $-\lambda$ is also an eigenvalue of $M_k$ with eigenvector $S^{\otimes k}\xi$. 
All nonzero eigenvalues of $M_k$ therefore occur in opposite pairs, and it has exactly
$(3^k-1)/2$ positive eigenvalues, all at least $\sqrt 2$.

Let $U\subseteq V(P_3^{\square k})$ have
$$
 |U|=\alpha(P_3^{\square k})+1=(3^k+3)/2.
$$
Ordering the eigenvalues of $M_k$ decreasingly, Cauchy interlacing yields
$$
 \lambda_{\max}(M_k[U])
 \ge\lambda_{(3^k-1)/2}(M_k)\ge\sqrt2.
$$
Since $M_k[U]$ is a signed adjacency matrix supported on
$P_3^{\square k}[U]$, its largest eigenvalue is at most its spectral radius, which is at most the maximum absolute row sum.  
The maximum absolute row sum of $M_k[U]$ equals $\Delta(P_3^{\square k}[U])$. Hence
$\Delta(P_3^{\square k}[U])\ge2$.

For the reverse inequality, give one $P_3$ factor the constant sign word
$+++$ and every other factor an alternating sign word.  The product coloring
has total sign sum of absolute value $3$, so its larger class has
$(3^k+3)/2$ vertices.  Its monochromatic edges occur only in the first
coordinate and have maximum degree $2$.  Hence $f(P_3^{\square k})\le 2$.
\end{proof}

Theorem~\ref{thm:all-path-intro} now follows from
Propositions~\ref{prop:mixed-box}, \ref{prop:odd-large}, and~\ref{prop:p3}.
Theorem~\ref{thm:even-paths-intro} is the special case of
Proposition~\ref{prop:mixed-box} in which all factors are even.

\section{The t-sensitivity of Cartesian products}\label{sec:fixed-excess}

We now turn to the parameter $f_t$ from \eqref{eq:def-ft}.
Lemma~\ref{lem:cube-cell-bound} supplies the lower bound
$\lceil\sqrt{k}\rceil$ for every admissible positive excess in a
product of $k$ bipartite graphs with perfect matchings. For the upper bound, we construct colorings with imbalance at least $2t$ and monochromatic degree at most $D$, then retain exactly $\alpha(G)+t$ vertices from the larger color class.

We begin with a flexible extension of the matched-bridge lift.  If
$\rho:V(B)\to\{-1,1\}$ is a bipartition character and
$\tau:V(B)\to\{0,1\}$, call the edges across the two $\tau$-fibers the
\emph{$\tau$-cut}.

\begin{theorem}\label{thm:matching-cut-ft}
Let $k\ge1$, and let $B_1,\ldots,B_k$ be connected bipartite
graphs with perfect matchings.  For each $i$, fix a bipartition character
$\rho_i:V(B_i)\to\{-1,1\}$ and a map
$\tau_i:V(B_i)\to\{0,1\}$ whose $\tau_i$-cut is a matching.  Put
$$
 \gamma_i=\left|\sum_{\tau_i(x)=0}\rho_i(x)\right|,
 \quad \Gamma=\prod_{i=1}^k\gamma_i.
$$
If every $\gamma_i$ is positive, then, for $1\le t\le\Gamma$,
$$
 f_t(B_1\square\cdots\square B_k)=\lceil\sqrt{k}\rceil.
$$
\end{theorem}

\begin{proof}
Set $G=B_1\square\cdots\square B_k$, $N=|V(G)|$, and
$D=\lceil\sqrt{k}\rceil$.  The product is bipartite and has a perfect
matching, so $\alpha(G)=N/2$.  Lemma~\ref{lem:cube-cell-bound} gives
the required lower bound $f_t(G)\ge D$.

If $|V(B_i)|=2n_i$, then $\gamma_i\le n_i$, which implies 
$\Gamma\le\prod_i n_i\le N/2$. Hence every value of $t$ in the statement is in the domain of $f_t$.

After reversing $\rho_i$ when necessary, assume that
$\sum_{\tau_i=0}\rho_i=\gamma_i$.  Balance gives
$\sum_{\tau_i=1}\rho_i=-\gamma_i$.  Define
$$
 \eta_i(x)=(-1)^{\tau_i(x)}\rho_i(x).
$$
Thus the $\eta_i$-sum on either $\tau_i$-fiber is $\gamma_i$.  Let
$\chi:\{0,1\}^k\to\{-1,1\}$ be the coloring from
Lemma~\ref{lem:cube-coloring}.  On $G$ put
$$
 \Psi(x_1,\ldots,x_k)
 =\chi(\tau_1(x_1),\ldots,\tau_k(x_k))
   \prod_{i=1}^k\eta_i(x_i).
$$
The sum factorizes and equals
$$
 \sum_{x\in V(G)}\Psi(x)
 =\Gamma\sum_{a\in\{0,1\}^k}\chi(a)=2\Gamma.
$$
Consequently the positive class has $N/2+\Gamma$ vertices.

Consider an edge in the $i$-th factor direction.  Off the $\tau_i$-cut, $\tau_i$ is unchanged and $\eta_i$ changes sign, so $\Psi$ changes sign. Across the cut, both $\tau_i$ and $\rho_i$ change, hence $\eta_i$ is unchanged; such an edge is monochromatic precisely when the corresponding
edge of $Q_k$ is monochromatic under $\chi$.  Since the cut is a matching, at a vertex $x=(x_1,\ldots,x_k)$ there is at most one cut edge in each factor direction. Each monochromatic incident edge corresponds to a distinct monochromatic edge at $(\tau_1(x_1),\ldots,\tau_k(x_k))$ in $Q_k$. Hence there are at most $D$ such edges.
Delete $\Gamma-t$ positive vertices.  Passing to an induced subgraph cannot increase maximum degree, so this produces a set of order $N/2+t$ and induced maximum degree at most $D$.
\end{proof}

For even paths, the best one-factor matching-cut imbalance has a simple closed form.

\begin{proposition}\label{prop:path-gamma}
Let $r\ge1$, let $P_{2r}=v_1\cdots v_{2r}$, and put
$\rho(v_i)=(-1)^{i+1}$.  Among all maps
$\tau:V(P_{2r})\to\{0,1\}$ whose $\tau$-cut is a matching,
$$
 \max_\tau\left|\sum_{\tau(v)=0}\rho(v)\right|
 =\left\lceil\frac r3\right\rceil.
$$
\end{proposition}

\begin{proof}
Fix an admissible map $\tau$, and set
$a_j=\tau(v_{2j})-\tau(v_{2j-1})$. Then
$$
\sum_{\tau(v)=0}\rho(v)=\sum_{j=1}^r a_j.
$$
If $a_j$ and $a_{j+1}$ are nonzero, the intervening edge
$v_{2j}v_{2j+1}$ is not in the $\tau$-cut, so
$a_{j+1}=-a_j$. Likewise, if $a_j,a_{j+2}\ne0$ and
$a_{j+1}=0$, the three intervening edges are not in the
cut, giving $a_{j+2}=-a_j$. Thus two consecutive nonzero terms whose indices differ by one or two have opposite signs. In particular, the nonzero terms within any three consecutive positions alternate, so their sum has absolute value at most one. Partitioning $[r]$ into consecutive blocks of
at most three indices yields
$$
\left|\sum_{j=1}^r a_j\right|
   \le \left\lceil\frac r3\right\rceil.
$$

For equality, let $\tau(v_1)=0$ and change its value
precisely across the edges $v_sv_{s+1}$ with
$s\equiv1\pmod 3$. These edges form a matching, and the
resulting sequence satisfies $a_j=1$ for
$j\equiv1\pmod 3$ and $a_j=0$ otherwise. Hence
$\sum_{j=1}^r a_j=\lceil r/3\rceil$.
\end{proof}

\begin{corollary}\label{cor:path-ft-range}
Let $k\ge1 $ and $r_1,\ldots,r_k\ge1$, and set
$\Gamma=\prod_{i=1}^k\left\lceil\frac{r_i}{3}\right\rceil$.
Then for $1\le t\le\Gamma$
$$
 f_t(P_{2r_1}\square\cdots\square P_{2r_k})
 =\lceil\sqrt{k}\rceil.
$$
\end{corollary}

A second construction uses sets of separated vertices of small degree.
The following lemma isolates the common argument used below.

\begin{lemma}
\label{lem:separated-blocks}
Let $D,q\ge1$ be integers, and let $X_1,\ldots,X_q$ be bipartite graphs
with bipartition characters
$\rho_j$ satisfying $\sum_{x\in V(X_j)}\rho_j(x)=0$.
For each $j$, let $Z_j\subseteq V(X_j)$ be a nonempty set on which
$\rho_j$ is constant.  Suppose that distinct vertices of $Z_j$ are
at distance at least three in $X_j$, and that every vertex of $Z_j$
has degree at most $D$ in $X_j$.  If $q\le D$, then
$X=X_1\square\cdots\square X_q$ has a coloring
$\sigma:V(X)\to\{-1,1\}$ with
$$
 \left|\sum_{x\in V(X)}\sigma(x)\right|
 =2\prod_{j=1}^q|Z_j|
$$
and monochromatic degree at most $D$ at every vertex.
\end{lemma}

\begin{proof}
Let $b_j$ be zero on $Z_j$ and one elsewhere, and set
$$
 F(x)=\prod_{j=1}^q b_j(x_j),\quad
 \rho(x)=\prod_{j=1}^q\rho_j(x_j),\quad
 \sigma(x)=\rho(x)(-1)^{F(x)}.
$$
Since each block is balanced and $\rho_j$ is constant on $Z_j$,
$$
 \left|\sum_{x\in V(X_j)}\rho_j(x)b_j(x)\right|=|Z_j|.
$$
The sums factor over the blocks.  Using $(-1)^{F(x)}=1-2F(x)$ and
$\sum_x\rho(x)=0$, we obtain
$$
 \left|\sum_x\sigma(x)\right|
 =2\left|\sum_x\rho(x)F(x)\right|
 =2\prod_{j=1}^q|Z_j|.
$$

Across every edge, $\rho$ changes sign.  Thus an edge is
monochromatic under $\sigma$ exactly when $F$ changes.
If at least two block values are zero, changing one coordinate
cannot change $F$.  If exactly one block value is zero, only
edges leaving that block's zero set can change $F$, and there
are at most $D$ such edges.  If no block value is zero, each
block contributes at most one edge on which $F$ changes:
a vertex adjacent to two distinct points of $Z_j$ would put
those points at distance at most two.  The total is therefore
at most $q\le D$.
\end{proof}

We next record two elementary facts about leaves.

\begin{lemma}\label{lem:pm-tree-leaves}
Let $T$ be a tree with a perfect matching and bipartition $(X,Y)$.  Then each
of $X$ and $Y$ contains a leaf.  Moreover, if $T\ne K_2$, then any two
distinct leaves of $T$ are at distance at least three.
\end{lemma}

\begin{proof}
Write $|X|=|Y|=n$.  If, say, $X$ contained no leaf, then every vertex of $X$ would have degree at least two. Thus
$$
 |E(T)|=\sum_{v\in X}d(v)\ge 2n,
$$
contrary to $|E(T)|=2n-1$.  The argument for $Y$ is the same.  Two leaves can be adjacent only when the connected tree is $K_2$.  Two leaves at distance two have a common neighbor, but then both incident leaf edges would be forced into every perfect matching, which is impossible.
\end{proof}

We next construct many even-weight binary words at mutual distance at least four. These words will specify separated tuples of leaves in the product.

\begin{lemma}\label{lem:even-code}
For every integer $D\ge1$, put $m=\lceil\log_2D\rceil$.
There is a set $\mathcal C\subseteq\{0,1\}^D$ of at least
$2^{D-m-1}$ even-weight words such that distinct words have
Hamming distance at least four.
\end{lemma}

\begin{proof}
Choose distinct vectors $h_1,\ldots,h_D\in\mathbb F_2^m$ and set
$$
 \mathcal C=\left\{c\in\mathbb F_2^D:
 \sum_{i=1}^D c_i=0,\quad
 \sum_{i=1}^D c_i h_i=0\right\}.
$$
These are at most $m+1$ linear constraints, so
$|\mathcal C|\ge2^{D-m-1}$.  Every word has even weight.
A word of weight two would imply $h_i=h_j$ for two distinct
indices, which is impossible.  Since $\mathcal C$ is a linear space,
the difference of two distinct words has weight at least four.
For $D=1$, take $\mathbb F_2^0=\{0\}$; the construction gives
$\mathcal C=\{0\}$.
\end{proof}

Theorem~\ref{thm:fixed-excess-intro} (i) gives an explicit fixed-excess range by using this code in one block of factors.

\begin{proof}[Proof of Theorem~\ref{thm:fixed-excess-intro} (i)]
Let $G=T_1\square\cdots\square T_k$, 
$m=\lceil\log_2D\rceil$ and $N=|V(G)|$.
Since $D=\lceil\sqrt{k}\rceil\le k$, the hypotheses imply
$t\le2^{D-1}\le2^{k-1}\le N/2$, so $f_t(G)$ is defined.
Moreover, $\alpha(G)=N/2$, and
Lemma~\ref{lem:cube-cell-bound} gives $f_t(G)\ge D$.

Fix bipartition characters $\rho_i$ of the factors.  By
Lemma~\ref{lem:pm-tree-leaves}, each $T_i$ has leaves
$a_i^0,a_i^1$ in opposite bipartition classes.
Since $D\le k\le D^2$, partition $[k]$ into nonempty blocks
$I_1,\ldots,I_q$ with
$$
 |I_1|=D,\quad |I_j|\le D\ (j\ge2),\quad q\le D.
$$
Such a partition exists because, after reserving $D$ indices,
the remaining at most $D(D-1)$ indices fit into at most
$D-1$ further blocks.  Write
$$
 X_j=\mathop{\square}_{i\in I_j}T_i,
 \quad \rho_{I_j}(x)=\prod_{i\in I_j}\rho_i(x_i).
$$
Each $X_j$ is balanced.

Identify $I_1$ with $[D]$, and choose $t$ words from the code
in Lemma~\ref{lem:even-code}.  For the chosen set $\mathcal C'$, put
$$
 Z_1=\{(a_i^{c_i})_{i\in I_1}:c\in\mathcal C'\}
 \subseteq V(X_1).
$$
The words have even weight, so $\rho_{I_1}$ is constant on $Z_1$.
The distance between two resulting tuples is at least the
Hamming distance between their words, and hence is at least four.
Every point of $Z_1$ has degree $D$, since all its coordinates
are leaves.  For $j\ge2$, let $Z_j$ consist of one all-leaf tuple
in $X_j$; its degree is $|I_j|\le D$.

Lemma~\ref{lem:separated-blocks} now gives a coloring of $G$
with imbalance $2t$ and monochromatic degree at most $D$.
Its larger class has $N/2+t=\alpha(G)+t$ vertices, proving
$f_t(G)\le D$.
\end{proof}

The next argument proves Theorem~\ref{thm:fixed-excess-intro} (ii) for tree products. In particular, it determines $f_2$ for products of even-order paths even when the matching-cut range above reaches only $t=1$.

\begin{proof}[Proof of Theorem~\ref{thm:fixed-excess-intro} (ii)]
Let $G=T_1\square\cdots\square T_k$ and $N=|V(G)|$. Then $\alpha(G)=N/2$ and
Lemma~\ref{lem:cube-cell-bound} gives $f_2(G)\ge D$.

Fix bipartition characters $\rho_i$ and relabel the factors so
that $T_1\ne K_2$.  Since $2\le D\le k\le D^2$, choose a
partition $I_1,\ldots,I_q$ of $[k]$ with
$$
 q\le D,\quad |I_j|\le D,\quad
 |I_1|=D,\quad 1,2\in I_1.
$$
Write $X_j=\mathop{\square}_{i\in I_j}T_i$ and
$\rho_{I_j}=\prod_{i\in I_j}\rho_i$.

Take distinct leaves $u,v$ of $T_1$.  If
$\rho_1(u)=\rho_1(v)$, use the same fixed leaf in every other
coordinate of two tuples in $X_1$.  If
$\rho_1(u)=-\rho_1(v)$, choose leaves $w,w'$ of $T_2$ in opposite
bipartition classes and use $(u,w)$ and $(v,w')$ in the first two
coordinates; keep every remaining coordinate fixed at a leaf.
In either case we obtain all-leaf tuples $z_1,z'_1$ with
$$
 \rho_{I_1}(z_1)=\rho_{I_1}(z'_1),\quad
 d_{X_1}(z_1,z'_1)\ge4.
$$
Indeed, Lemma~\ref{lem:pm-tree-leaves} gives $d_{T_1}(u,v)\ge3$.
In the first case this distance is even and hence at least four;
in the second case the second coordinate contributes at least one
more to the product distance.

Set $Z_1=\{z_1,z'_1\}$, and for each $j\ge2$ let $Z_j$ consist
of one all-leaf tuple in $X_j$.  Every point of $Z_j$ has degree
$|I_j|\le D$.  Lemma~\ref{lem:separated-blocks} gives a coloring
with imbalance four and monochromatic degree at most $D$.
Its larger class has $N/2+2$ vertices, proving the upper bound.
\end{proof}

\begin{corollary}\label{cor:even-path-ftwo}
Let $k\ge2$ and $r_1,\ldots,r_k\ge1$.  If some $r_i\ge2$, then
$$
 f_2(P_{2r_1}\square\cdots\square P_{2r_k})
 =\lceil\sqrt{k}\rceil.
$$
Thus the only product of at least two even paths not covered by this corollary
is the hypercube $P_2^{\square k}=Q_k$.
\end{corollary}

The restriction to at least two factors in Corollary~\ref{cor:even-path-ftwo} is necessary. For a single even-order path, the values are as follows.

\begin{proposition}\label{prop:one-path-ftwo}
For $r\ge 2$,
$$
 f_2(P_{2r})=
 \begin{cases}
  2,& r=2,3,\\
  1,& r\ge 4.
 \end{cases}
$$
\end{proposition}

\begin{proof}
The elementary path formula
$\operatorname{diss}(P_n)=\lceil2n/3\rceil$ follows by partitioning the path into consecutive triples, with a final block of one or two vertices:
a maximum-degree-one set contains at most two vertices from each triple, and the repeating word $110$ (where $1$ means selected) attains the bound including the remainder one or two vertices.
For $n=4,6$, $\operatorname{diss}(P_n)=\lceil2n/3\rceil < n/2+2=\alpha(P_n)+2$, so $f_2(P_n)\ge2$; the reverse inequality follows from
$\Delta(P_n)=2$.  If $n=2r$ with $r\ge4$, then
$$
 \operatorname{diss}(P_{2r})=\left\lceil\frac{4r}{3}\right\rceil
 \ge r+2=\alpha(P_{2r})+2,
$$
so a maximum dissociation set contains a subset of exactly
$r+2=\alpha(P_{2r})+2$ vertices and $f_2(P_{2r})\le1$. A set of more than $\alpha(P_{2r})$ vertices cannot be independent.
\end{proof}

\section{The 2-sensitivity of the hypercube}\label{sec:cube-excess-two}

Corollary~\ref{cor:even-path-ftwo} does not cover products in which every factor is $K_2$. We now treat these products using the Boolean notation from Section~\ref{sec:cube}.

\begin{proposition}\label{prop:cube-excess-two-large}
$f_2(Q_n)=\lceil\sqrt n\rceil$ for $n\ge 10$.
\end{proposition}

\begin{proof}
For $n\ge 10$, $D=\lceil\sqrt n\rceil\ge 4$ and $2^{D-\lceil\log_2 D\rceil-1}\ge 2$. Applying Theorem~\ref{thm:fixed-excess-intro} (i) with every factor equal to $K_2$ gives the upper bound, and Huang's theorem gives the lower bound.
\end{proof}

\begin{proposition}\label{prop:cube-excess-two-small}
For $n\in\{2,3,5,6,7,8\}$,
$f_2(Q_n)=\lceil\sqrt n\rceil$.
In contrast, $f_2(Q_4)=3$.
\end{proposition}

\begin{proof}
Huang's theorem gives the required lower bounds for $n \neq 4$. For $Q_2$, take the whole cube.  For $Q_3$, delete an antipodal pair; the remaining six vertices induce a cycle.

We use Lemma~\ref{lem:parity-twist-excess} for the other upper bounds.
Define $F_5:\{0,1\}^5\to\{0,1\}$ by declaring $F_5(x)=1$ when $|x|\le2$ or when
$$
\operatorname{supp}(x)\in\mathcal T=
 \bigl\{\{1,2,5\},\{1,3,4\},\{2,3,4\},\{3,4,5\}\bigr\}.
$$
The four triples in $\mathcal T$ cover every pair of $[5]$, and no
four-subset contains all four triples.  A zero of weight three has exactly three neighbors in the one-set; a zero of weight four has at most three selected triple neighbors; a zero of weight five has no neighbors in the one-set.  At a vertex of weight two, $F_5=1$, and at most two of its three neighbors of weight three have value zero, since every pair is contained in a triple of $\mathcal T$. At a selected triple, only its two neighbors of weight four have value zero. Therefore
$$
s_0(F_5)\le3,\quad s_1(F_5)\le2,\quad A(F_5)=1-5+10-4=2.
$$
This settles $n=5$.

Define $g_3:\{0,1\}^3\to\{0,1\}$ by setting $g_3(x)=1$ exactly when $|x|\le 1$. Then
$$
 s_0(g_3)=s_1(g_3)=2, \quad A(g_3)=1-3=-2.
$$
The Boolean $\operatorname{OR}_3$ function has alternating sum $A(\operatorname{OR}_3)=-1$ and $s_0(\operatorname{OR}_3)=3$, $s_1(\operatorname{OR}_3)=1$. Hence, the function
$$
 F_6=g_3\wedge\operatorname{OR}_3
$$
on two disjoint three-coordinate blocks has $A(F_6)=2$ and $s(F_6)\le 3$.  This settles $n=6$.

On four variables, fix $P\in\binom{[4]}2$ and define $g_4=0$ exactly on the vectors of weight at most one and on the vector with support $P$.
Then
$$
s_0(g_4)\le 3,\quad s_1(g_4)\le 2,\quad  A(g_4)=2.
$$
Indeed, a zero singleton outside $P$ has three one-neighbors, while all other zero cases have at most two; a one of weight two has at most its two lower zero-neighbors. Ones of weights three and four have sensitivity at most one and zero, respectively. It follows that
$$
 F_7=g_4\wedge\operatorname{OR}_3
 \text{ and }
 F_8=F_5\wedge\operatorname{OR}_3
$$
have $|A(F_7)|=|A(F_8)|=2$ and sensitivities at most
three. They settle $n=7,8$.

For the upper bound in $Q_4$, let $W=\{000,001,010,101,110\}\subseteq V(Q_3)$.
Then $|W|=5$ and $\Delta(Q_3[W])=2$, so
$W\times\{0,1\}\subseteq Q_4$ has ten vertices and induced maximum degree
three.

It remains to exclude induced maximum degree two.  Let $E,O$ be the two
parity classes of $Q_4$, let $U\subseteq V(Q_4)$ have size ten, and suppose
to the contrary that $\Delta(Q_4[U])\le2$.  Interchanging $E$ and $O$
if necessary, put
$$
 X=U\cap O,
 \quad a=|X|\le5,
 \quad Y=E\setminus U.
$$
Since $|U|=10$ and $|E|=8$, we have $|Y|=a-2$. Each vertex of $X$ has
at least two neighbors in $Y$, so
$$
 2a\le e(X,Y)\le 4(a-2), 
$$
which implies $a=4$ or $a=5$.

If $a=4$, equality holds throughout, so every vertex of $X$ is a common
neighbor of the two vertices of $Y$.  This is impossible: two distinct
vertices in one parity class of $Q_4$ have two common neighbors when their
distance is two and none when their distance is four.

If $a=5$, then $|Y|=3$ and at least five odd vertices have two or more neighbors in $Y$.  For $v\in O$, write $d_Y(v)=|N(v)\cap Y|$.  Counting pairs consisting of an odd vertex and two of its neighbors in $Y$ gives
$$
 \sum_{v\in O}\binom{d_Y(v)}2
 =\sum_{\{y,y'\}\in\binom Y 2}|N(y)\cap N(y')|.
$$
At least five vertices of $X$ contribute one each, so the left-hand side is at least five. If $Y$ contains an antipodal pair, that pair contributes zero on the right, and each of the other two pairs contributes at most two, a contradiction.  Hence all three
pairs in $Y$ are at distance two. By translating one vertex of $Y$ to $0000$, the other two vertices have weight two and supports intersecting in exactly one coordinate. Permuting the coordinates, we may therefore assume that
$$
 Y=\{0000,1100,1010\}.
$$
The multiset $\{d_Y(v):v\in O\}$ is $\{3,2,2,2,1,1,1,0\}$,
and the odd vertices with at least two neighbors in $Y$ are
$1000,0100,0010,1110$.  Thus only four odd vertices have at least two
neighbors in $Y$, again a contradiction. Therefore $f_2(Q_4)=3$.
\end{proof}

Combining Propositions~\ref{prop:cube-excess-two-large}
and~\ref{prop:cube-excess-two-small} determines $f_2(Q_n)$ in every dimension $n\ge2$ except $n=9$. In that dimension, our bounds are
$$
 3\le f_2(Q_9)\le4.
$$
For the upper bound, take the $130$-vertex set of maximum degree three in
$Q_8$ supplied above, delete one vertex, and take its Cartesian product
with $K_2$.  The result has $258=2^8+2$ vertices and maximum degree at most
four. The exact value in dimension nine is not determined here.

\section{Concluding remarks}\label{sec:discussion}

For $r\ge1$, Zeng and Hou~\cite{ZeHo2024} proved
$$
 f(P_{2r}^{\square k})
 \ge
 \left\lceil
  2\sin\!\left(\frac{\pi}{4r+2}\right)\sqrt{k}
 \right\rceil.
$$
When $r\ge2$, the coefficient of $\sqrt{k}$ is strictly less than $1$.
Theorem~\ref{thm:even-paths-intro} therefore shows that this spectral lower
bound is not asymptotically sharp and, in particular, answers negatively the
tightness question posed in~\cite{ZeHo2024}.  The obstruction forcing the
larger value is local: the perfect matchings partition the entire product
into cubes, one of which must receive a strict majority.

The upper-bound lift is equally local.  Lemma~\ref{lem:bridge-gate} replaces
each matched factor by one binary coordinate while preserving an imbalance of
two in the final coloring.  The bridge assumption ensures that all remaining
edges reverse color.  It is natural to ask how far this hypothesis can be
relaxed.

\begin{problem}\label{prob:matched-factors}
Characterize the connected bipartite graphs $B$ with perfect matchings such
that
$$
 f(B^{\square k})=\lceil\sqrt{k}\rceil
 \qquad\text{for every }k\ge1.
$$
\end{problem}

Theorem~\ref{thm:matching-cut-ft} already extends the matched-bridge
construction to factors admitting a matching cut with nonzero bipartition
imbalance.  It remains to determine how far this sufficient condition is
from being necessary.  More generally, one may ask which classes of factors
have the same property for products of distinct members.

The fixed-excess parameter agrees with the prescribed-density parameter
of Garc\'ia-Marco and Knauer~\cite{GaMK2022} at
$\beta=(\alpha(G)+t)/|V(G)|$; this follows by deleting vertices if necessary.
Zeng and Hou~\cite{ZeHo2024} recorded Huang's question about
$\alpha(G)+t$, and Kriloff and Tolman~\cite{KrTo2026} call the
corresponding parameter $k$-sensitivity.
The two-sided imbalance parameter used
in~\cite{GaMK2022} is related but imposes a stronger condition.
For a bipartite graph $G$, write
$$
 \iota_D(G)=\max\bigl\{|A|-|B|:
 V(G)=A\sqcup B,\ 
 \Delta(G[A])\le D,\Delta(G[B])\le D\bigr\}.
$$
For a bipartite graph with a perfect matching, $\alpha(G)=|V(G)|/2$,
and hence
$$
 \iota_D(G)\ge 2t\quad\Longrightarrow\quad f_t(G)\le D.
$$
The converse need not hold, even in this graph class.

\begin{example}\label{ex:one-sided-versus-two-sided}
Let $G$ have bipartition $\{u,a,b,c\}\sqcup\{v,x,y,z\}$ and edge set
$$
 E(G)=\{uv,ux,uy,uz,va,vb,vc,ax,ay,by,cz\}.
$$
Then $f(G)=1$ but $\iota_1(G)=0$.
\end{example}
\begin{proof}
The edges $uv,ax,by,cz$ form a perfect matching, so $\alpha(G)=4$.
The set $\{a,b,c,x,z\}$ induces two independent edges and an isolated
vertex, giving $f(G)=1$.
A set of five vertices inducing maximum degree at most one cannot contain
$u$: it would have to contain all of $a,b,c$ and exactly one neighbor of
$u$, each choice producing a vertex of degree at least two.
The same argument excludes $v$.
The remaining six vertices induce the path $xayb$ and the edge $cz$.
Thus the only possible five-vertex sets are
$\{a,b,c,x,z\}$ and $\{b,c,x,y,z\}$, whose complements both induce $P_3$.
There is no larger set inducing maximum degree at most one.
Consequently no such two-sided coloring has unequal class sizes, whereas
the bipartition itself gives imbalance zero.
\end{proof}

Theorem~\ref{thm:fixed-excess-intro} (i) determines $f_t$ for products of trees with perfect matchings whenever
$$
 1\le t\le2^{D-\lceil\log_2D\rceil-1},
 \quad D=\lceil\sqrt{k}\rceil.
$$
The coding construction supplies many vertices of the same bipartition
sign with disjoint neighborhoods; the graph-theoretic construction then
converts them into the required excess while preserving the degree bound.
Theorem~\ref{thm:matching-cut-ft} gives a second exact regime, while
Theorem~\ref{thm:fixed-excess-intro} (ii) removes the asymptotic restriction at excess two as soon as one factor is not $K_2$.

The remaining questions concern products with few factors.  Our results
determine $f_2(Q_n)$ for all $n\ge2$ except $n=9$.

\begin{problem}\label{prob:excess}
Determine whether $f_2(Q_9)=3$ or $4$.  More generally, sharpen the factor threshold in Theorem~\ref{thm:fixed-excess-intro} (i), and determine $f_t$ in the remaining finite-factor cases for hypercubes, Cartesian products of trees with perfect matchings, and Cartesian products of even-order paths.
\end{problem}

A cube partition supplies a lower bound, but this bound need not be sharp. For instance, $C_4\cong Q_2$ and hence $C_4^{\square3}\cong Q_6$, whose sensitivity is three, while partitioning along one perfect matching in each $C_4$ gives only the lower bound two. It remains to identify when a cube partition gives a sharp bound and when a coloring of its cubes extends to the whole product with the same degree bound. These questions differ because our lifts control both color classes, whereas $f$ is defined using only one vertex set.

\section*{Acknowledgement}

The research of Zhen-Mu Hong is supported by Natural Science Foundation of China (No. 12371338) and Outstanding Youth Scientific Research Projects of Anhui Provincial Department of Education (No. 2022AH030073). The research of Zheng-Jiang Xia is supported by
Key Projects in Natural Science Research of Anhui Provincial Department of Education (No. 2023AH050268).


\begin{thebibliography}{99}

\bibitem{AlZh2020}
N.~Alon and K.~Zheng,
Unitary signings and induced subgraphs of Cayley graphs of $\mathbb Z_2^n$,
\emph{Adv. Comb.} (2020), Paper No. 11, 12 pp. \href{https://doi.org/10.19086/aic.17912}{https://doi.org/10.19086/aic.17912}

\bibitem{AFGK2026}
S.~Asensio, Y.~Filmus, I.~Garc\'ia-Marco, and K.~Knauer,
Sensitivity and Hamming graphs,
\emph{J. Graph Theory} \textbf{112} (2026), 296--305. \href{https://doi.org/10.1002/jgt.70034}{https://doi.org/10.1002/jgt.70034}

\bibitem{BJKST2013}
B.~Bre\v{s}ar, M.~Jakovac, J.~Katreni\v{c}, G.~Semani\v{s}in, and A.~Taranenko,
On the vertex $k$-path cover,
\emph{Discrete Appl. Math.} \textbf{161} (2013), 1943--1949. \href{https://doi.org/10.1016/j.dam.2013.02.024}{https://doi.org/10.1016/j.dam.2013.02.024}

\bibitem{CFGS1988}
F.~R.~K.~Chung, Z.~F\"uredi, R.~L.~Graham, and P.~Seymour,
On induced subgraphs of the cube,
\emph{J. Combin. Theory Ser. A} \textbf{49} (1988), 180--187. \href{https://doi.org/10.1016/0097-3165(88)90034-9}{https://doi.org/10.1016/0097-3165(88)90034-9}

\bibitem{GaMK2022}
I.~Garc\'ia-Marco and K.~Knauer,
On sensitivity in bipartite Cayley graphs,
\emph{J. Combin. Theory Ser. B} \textbf{154} (2022), 211--238. \href{https://doi.org/10.1016/j.jctb.2022.01.002}{https://doi.org/10.1016/j.jctb.2022.01.002}

\bibitem{HoLL2021}
Z.-M.~Hong, H.-J.~Lai, and J.-B.~Liu,
Induced subgraphs of product graphs and a generalization of Huang's theorem,
\emph{J. Graph Theory} \textbf{98} (2021), 285--308. \href{https://doi.org/10.1002/jgt.22692}{https://doi.org/10.1002/jgt.22692}

\bibitem{Huang2019}
H.~Huang,
Induced subgraphs of hypercubes and a proof of the Sensitivity Conjecture,
\emph{Ann. of Math.} \textbf{190} (2019), 949--955. \href{https://doi.org/10.4007/annals.2019.190.3.6}{https://doi.org/10.4007/annals.2019.190.3.6}

\bibitem{KrTo2026}
C.~Kriloff and J.~Tolman,
Graph sensitivity under join and decomposition,
preprint, arXiv:2512.19915v4, 2026. \href{https://arxiv.org/abs/2512.19915v4}{https://arxiv.org/abs/2512.19915v4}

\bibitem{JaTa2013}
M.~Jakovac and A.~Taranenko,
On the $k$-path vertex cover of some graph products,
\emph{Discrete Math.} \textbf{313} (2013), 94--100. \href{https://doi.org/10.1016/j.disc.2012.09.010}{https://doi.org/10.1016/j.disc.2012.09.010}

\bibitem{LiZu2018}
Z.~Li and L.~Zuo,
The $k$-path vertex cover in Cartesian product graphs and complete bipartite graphs,
\emph{Appl. Math. Comput.} \textbf{331} (2018), 69--79. \href{https://doi.org/10.1016/j.amc.2018.03.008}{https://doi.org/10.1016/j.amc.2018.03.008}


\bibitem{Tika2022}
A.~Tikaradze,
Induced subgraphs of powers of oriented cycles,
\emph{Linear Multilinear Algebra} \textbf{70} (2022), 3301--3303. \href{https://doi.org/10.1080/03081087.2020.1836116}{https://doi.org/10.1080/03081087.2020.1836116}

\bibitem{ZeHo2024}
J.~Zeng and X.~Hou,
On induced subgraph of Cartesian product of paths,
\emph{J. Graph Theory} \textbf{107} (2024), 169--180. \href{https://doi.org/10.1002/jgt.23116}{https://doi.org/10.1002/jgt.23116}

\end{thebibliography}
\end{document}